\documentclass[12pt]{amsart}

\usepackage[margin=1.15in]{geometry}
\usepackage{amscd,amssymb, amsmath, wasysym, mathrsfs, mathtools, bbold, color,stmaryrd}
\usepackage{graphicx}
\usepackage[all, cmtip]{xy}

\usepackage{url}

\usepackage[textsize=tiny,textwidth=25mm]{todonotes}

\definecolor{hot}{RGB}{65,105,225}

\usepackage[pagebackref=true,colorlinks=true, linkcolor=hot ,  citecolor=hot, urlcolor=hot]{hyperref}

\theoremstyle{plain}
\newtheorem{theorem}{Theorem}[section]
\newtheorem{prop}[theorem]{Proposition}

\newtheorem{lm}[theorem]{Lemma}

\newtheorem{cor}[theorem]{Corollary}

\newtheorem{lemma}[theorem]{Lemma}
\newtheorem{thrm}[theorem]{Theorem}
\theoremstyle{definition}

\newtheorem{rmk}[theorem]{Remark}

\newtheorem*{ex*}{Example}

\newcommand\cO{{\mathcal O}}

\newcommand\bQ{{\mathbf{Q}}}
\newcommand\bZ{{\mathbf{Z}}}
\newcommand\bR{{\mathbf{R}}}
\newcommand\bC{{\mathbf{C}}}

\newcommand\bN{{\mathbf{N}}}

\DeclareMathOperator{\homo}{Hom}

\DeclareMathOperator{\spec}{Spec}

\def\ra{\rightarrow}
\def\SL{{\mathrm{SL}}}

\title{On representation zeta functions for special linear groups}

\author{Nero Budur}
\author{Michele Zordan}
\address{KU Leuven, Celestijnenlaan 200B, B-3001 Leuven, Belgium} 
\email{nero.budur@kuleuven.be}
\email{michele.zordan@kuleuven.be}

\begin{document}

\date{}

\begin{abstract}
We prove that the  numbers of irreducible $n$-dimensional  complex continuous representations of the special linear groups over $p$-adic integers grow slower than $n^2$. We deduce that the abscissas of convergence of the representation zeta functions of the special linear groups over the ring of integers are bounded above by $2$. In order to show these results we prove also that if $G$ is a connected, simply connected, semi-simple algebraic group defined over the field of rational numbers, then the $G$-representation variety of the fundamental group of a compact Riemann surface of genus $n$ has rational singularities if and only if the $G$-character variety has rational singularities. 
\end{abstract}
\maketitle

\section{Introduction}
\thispagestyle{empty}
In this note we prove:

\begin{thrm}\label{thrmMain1}
There exists a constant $C$ such that the number of isomorphism classes of continuous irreducible $n$-dimensional complex representations of $\mathrm{SL}_d(\bZ_p)$ is less than $Cn^2$ for any positive integer $d$ and prime number $p$. 
\end{thrm}

Larsen and Lubotzky conjectured in \cite[Section 11]{LL} the existence of universal upper bounds. This was then shown to be at least $Cn^{22}$ by Aizenbud and Avni in \cite{AA}. The present paper improves this latter result.

For a topological group $\Gamma$, let $r_n(\Gamma)$ be the number of isomorphism classes of continuous irreducible $n$-dimensional complex representations. Assuming that this number is finite  for every $n$, one defines the {\it representation zeta function} of $\Gamma$ to be the Dirichlet generating function
$$
\zeta_{\Gamma}(s) = \sum_{n\ge 1} r_n(\Gamma)n^{-s}.
$$
It is known that if such functions converge then they converge in a open complex half-plane. The smallest $s_0\in\bR$ such that $\zeta_{\Gamma}(s)$ converges for complex numbers $s$ with real part $Re(s) > s_0$ is called the {\em abscissa of convergence} of $\zeta_\Gamma(s)$ and denoted by $\alpha_\Gamma$. If $\zeta_{\Gamma}(s)$ does not converge, we say that $\alpha_\Gamma = \infty$. 

Let $K$ be a non-archimedean local field containing $\bQ$. By Jaikin-Zapirain \cite{J}, if $\Lambda$ is a compact open subgroup of $\SL_d(K)$ which is also a uniform pro-$p$ group, then $\zeta_{\Lambda}(s)$ is a rational function in $p^{-s}$ whose denominator is a product of binomials of the form $1-p^{a-bs}$ for some integers $a, b$. This implies that $\alpha_\Lambda$ is finite and rational.

 Since every compact open  subgroup $\Gamma$ of $\SL_d(K)$ contains a finite index subgroup like $\Lambda$ above, also $\alpha_\Gamma$ is finite and rational (and equal to $\alpha_\Lambda$ by \cite[Lemma 2.2]{LM}). Moreover, if
\[
	R_N(\Gamma) = \sum_{n = 1}^{N} r_n(\Gamma)
\]
for $N\in \bN$, then $\log (R_N(\Gamma))/ \log N$ tends to $\alpha_\Gamma$ as $N$ tends to infinity. Theorem \ref{thrmMain1} is then true for any compact open subgroup $\Gamma$ of $\SL_d(K)$, and it is  a consequence of the following more general theorem:
\begin{thrm}
\label{thrmMain}
Let $d$ be a positive integer. Let $K$ be a non-archimedean local field containing $\bQ$. Let $\Gamma$ be a compact open subgroup of $\mathrm{SL}_d(K)$. Then $\alpha_\Gamma < 2$.
\end{thrm}

In \cite{AA}, the bound $\alpha_\Gamma<22$ had been obtained with $\Gamma$ as in the theorem, along with other bounds for compact open subgroups for all semi-simple algebraic groups defined over ${\bQ}_p$. Our result complements a rather short list of known results besides \cite{AA}. When $\Gamma=\SL_d(\bZ_p)$,
$$
\alpha_{\Gamma}\text{ is }\left\{
						\begin{array}{lll}
							\ge 1/15 & & \text{by \cite{LL},}\\
							=1 & \text{ if } d=2 & \text{by \cite{J} and \cite{AKOV12}},\\
							=2/3 & \text { if }d=3  \text{ and } p > 3 & \text{by \cite{AKOV13}},\\
							=1/2 & \text { if }d=4 \text{ and } p > 2 & \text{by \cite{Z}.}
\end{array}\right.
$$

%
The proof of Theorem \ref{thrmMain} is based on Aizenbud-Avni \cite{AA}, together with the additional observation by Bellamy-Schedler \cite{BS} that the  $\SL_d$-character variety of a compact Riemann surface admits algebraic symplectic singularities. The new idea of this paper is that it is possible to use the machinery of \cite{AA} to deduce upper-bounds for the abscissa of convergence by looking at the singularities of the $\SL_d$-character variety rather than at those of the $\SL_d$-representation variety. As we will see in the next section, the methods used by \cite{AA} are only capable of computing the values of the representation zeta function at even natural numbers, therefore the present result is the best possible upper-bound obtainable with their techniques.

Theorem \ref{thrmMain} has the following notable corollaries as pointed out by Aizenbud-Avni. 

\begin{cor}\label{corRat} 
 Let $\pi_1(\Sigma_n)$ be the fundamental group at a fixed base point of a compact Riemann surface of genus $n$. Then the representation variety $$\homo (\pi_1(\Sigma_n),\SL_d)$$
has rational singularities over $\bQ$  for $n\ge 2$.
\end{cor}

This follows directly from Theorem \ref{thrmMain} together with \cite[Theorem IV]{AA}. The statement for the cases $n\ge 12$ was shown in \cite[Theorem VIII]{AA}.

\begin{cor} 
Let $d$ be a positive integer. Then there exists a finite set $S$ of prime numbers such that for every 
global field $k$ of characteristic not in $S$ and every finite set $T$ of places of $k$ containing all archimedean ones,  
$$\alpha (\SL_d(O_{k,T}))\le 2,$$
where 
$$O_{k,T}=\{x\in k\mid \forall v\not\in T, \|x\|_v\le 1\},$$
with the exception of the case $d=2$ and $k$ equal to $\bQ$ or an imaginary quadratic extension of $\bQ$.
In particular, $\alpha (\SL_d(\bZ))\le 2$ for $d\ge 3$.
\end{cor}

This follows directly from Corollary \ref{corRat} together with \cite[Theorem II]{AA2}, where a condition was phrased in terms of the congruence subgroup property for $\SL_d(O_{k,T})$. This property holds except  for $d=2$ and $k$ equal to $\bQ$ or an imaginary quadratic extension of $\bQ$ by \cite{BMS}, hence the statement of the above corollary. The statement of the corollary with $\alpha\le 22$ was shown in \cite{AA2}. Here $\alpha (\SL_d(\bZ))$ is defined to be the abscissa of convergence of the representation zeta function of the group $\SL_d(\bZ)$, which in this case enumerates all irreducible complex representations and not only the continuous ones.

Our argument for Theorem \ref{thrmMain}  extends to other semi-simple algebraic groups if the associated character variety  has rational singularities, that is, if the result of \cite{BS}  generalizes.  In this case, Corollary \ref{corRat} also generalizes to:

\begin{prop}
\label{prop:semisimple} 
Let $G$ be a connected, simply connected, semi-simple algebraic group defined over $\bQ$. Let $n\ge 2$. Then  the representation variety $$\homo (\pi_1(\Sigma_n),G)$$
has rational singularities  if and only if  the character variety
$$
\homo (\pi_1(\Sigma_n),G)\sslash G
$$
has rational singularities.
\end{prop}
See Section \ref{secPf} for the definitions. The ``only if" direction is due to Boutot \cite{Bo}. The ``if" direction follows as in the proof of Corollary \ref{corRat}, by noting that the only new thing we have used is the fact that the $\SL_d$-character variety has rational singularities over $\bQ$. If $n\ge 374$, both spaces in the corollary admit rational singularities regardless of $G$, by \cite[Theorem B]{AA}. Hence the corollary could be a step to get rid of this uniform lower bound.

Throughout the article, for a scheme $X$ and an algebra $A$, we denote by $X(A)$ the set of morphisms of $\bZ$-schemes $\spec(A)\ra X$.  A variety over a field $k$ is a $k$-scheme of finite type that is separated, reduced, and irreducible.
 
\medskip
 
 {\it Acknowledgement.} The authors thank Benjamin Martin for his comments on an earlier version of this paper. The second author would like to thank Avraham Aizenbud and Nir Avni for answering a question on their techniques to pass from a bound for the local abscissa to a bound for the global abscissa. The authors also thank the referees for very important comments. The first author was partly sponsored by a research grant STRT/1 3/005 and a Methusalem grant METH/15/026 from KU Leuven,  and by the research projects \allowbreak G0B2115N and G0F4216N from the Research Foundation - Flanders (FWO). The second author is currently supported by the research project G.0792.18N of FWO.

\section{Reduction to finite volumes}


Let $G$ be an algebraic group defined over $\bZ$. Let $K$ be a non-archimedean local field containing $\bQ$. The set of $K$-rational points $G(K)$ of $G$ is naturally a $K$-analytic manifold, since in characteristic zero every algebraic group is smooth.

Let $\Sigma_n$ be a compact Riemann surface of genus $n$, and let $\pi_1(\Sigma_n)$ be the fundamental group of $\Sigma_n$ based at a fixed point. For a compact open subgroup $\Gamma$ of the $K$-analytic group $G(K)$, define
$$
R(n,\Gamma)^o:=\{\rho\in \homo (\pi_1(\Sigma_n), \Gamma)\mid \text{ the closure of the image of }\rho\text{ is open}\}.
$$
In this case, the quotient $$R(n,\Gamma)^o/\Gamma$$ by the conjugation action admits the structure of a  $K$-analytic  manifold, together with a specific volume form $v_\Gamma$ by \cite[Section 4.3]{AA}.  We will use the following fascinating relation between the volume of $R(n,\Gamma)^o/\Gamma$ and values of $\zeta_\Gamma(s)$ due to Aizenbud-Avni:

\begin{thrm}\label{thrmVI}{\rm{(}\cite[Theorem VI]{AA}\rm{)} }
Let $G$ be a connected, simply connected, semi-simple algebraic group over $\bQ$. Let $K$ be a non-archimedean local field containing $\bQ$. Let $\Gamma$ be a compact  open subgroup of $G(K)$. Then there exists a non-zero constant $c_{G,\Gamma}$ such that
$$
\int_{R(n,\Gamma)^o/\Gamma} |v_\Gamma | = c_{G,\Gamma}\cdot \zeta_\Gamma (2n-2)
$$
for $n\ge 2$, if any the two sides of the equation is finite. In other words, if one side converges, the other side converges too, and equals the value prescribed by the equation.
\end{thrm}
\begin{proof} Strictly speaking, the statement of Theorem VI in \cite{AA} is slightly differently phrased. There are two differences with respect to the way we stated it: 

The first difference is that a certain ``FRS property" is assumed in {\it loc.\ cit.} This property was used by them to apply \cite[Theorem 3.4]{AA} to guarantee the finiteness on both sides of the equality on top of p.303 of \cite{AA}. If we assume however that one of those two sides is finite, the other is finite and equal to the former side since the equality is guaranteed by \cite[Corollary 3.6 (2)]{AA}. 

The extra assumption in \cite[Theorem VI]{AA} that the  characteristic of the residue field of $K$ must be $>2$ was not used in their proof, hence we can skip it.

Similarly, the first equality in their proof of \cite[Theorem VI]{AA} on p.302 uses \cite[Proposition I]{AA} and it is stated with a finiteness assumption on $\zeta_\Gamma (2n-2)$. This finiteness assumption can be dropped, in the sense that the equality is true as soon as one of the sides is finite, as it is shown in their proof of \cite[Proposition I]{AA}. 
\end{proof}

If $\zeta_\Gamma (s)$ converges for $s=s_0$, then it converges for $Re(s)>Re(s_0)$ and the function thus defined is holomorphic, \cite[Corollary 1, p. 66]{S}. In particular, to show Theorem \ref{thrmMain} it is enough by Theorem \ref{thrmVI} to prove the following, which is the main technical result of this article:

\begin{thrm}\label{mainTec} With the assumptions as in Theorem \ref{thrmVI}, the volume of $R(n,\Gamma)^o/\Gamma$ with respect to $v_\Gamma$ is finite when $\Gamma$ is a compact open subgroup of $\mathrm{SL}_d(K)$ and $n\ge 2$. 
\end{thrm}

The proof of this theorem will be based on the essential features of the form $v_\Gamma$   which we recall next.

\section{Proof of Theorem \ref{mainTec}}\label{secPf}

Let $G=\SL_d$ for a positive integer $d$. This an algebraic group defined over $\bZ$.
  Let
$$
R(n,G) := \homo (\pi_1(\Sigma_n),G)
$$
be the representation variety. This is an affine scheme of finite type defined over $\bZ$ which is a fine moduli space for the functor from commutative rings to sets given by
\begin{align*}
A\mapsto R(n,G)(A)&=\homo (\pi_1(\Sigma_n),G(A))  =\\
&= \{(g_1,h_1,\ldots, g_n,h_n)\in G(A)^{\times 2n}\mid [g_1,h_1]\cdots[g_n,h_n]=1 \}.
\end{align*}
There is a group scheme action of $G$ on $R(n,G)$ defined over $\bZ$ given by conjugation. Let
$$M(n,G):=R(n,G)_\bQ\sslash G_\bQ=\spec A^{G_\bQ}$$ 
be the categorical quotient, where the subscript $(\_)_\bQ$ denotes the base change to $\bQ$,  $A$ is the ring of global sections of the structure sheaf of $R(n,G)_\bQ$, and $A^{G_\bQ}$ is the ring of invariants
$$
A^{G_\bQ}:=\{a\in A\mid m(a)=a\otimes 1\},
$$
where $m:A\ra A\otimes_\bQ \Gamma(G_\bQ,\cO_{G_\bQ})=A\otimes_\bQ(\Gamma(G,\cO_G)\otimes\bQ)$ is the algebra homomorphism giving the action $G_\bQ\times R(n,G)_\bQ\ra R(n,G)_\bQ$. 

By \cite[Theorem 1.1 on p.27]{Mu}, $M(n,G)$ is also a universal categorical quotient. In particular, for a field extension $\bQ\subset k$, the base change
$$
M(n,G)_k=M(n,G)\times_{\spec \bQ} \spec k =\spec (A^{G_\bQ}\otimes_\bQ k)
$$
is the spectrum of the $k$-algebra
$$
(A\otimes_\bQ k)^{G_{ k}},
$$
where $G_{ k}=G\times_{\spec\bQ} \spec k$. In particular  
$
M(n,G)_\bC 
$
is an affine $\bC$-scheme of finite type, equal to the spectrum of the $\bC$-algebra of invariants
$
(A\otimes\bC)^{G(\bC)},
$ 
commonly denoted $R(n,G(\bC))\sslash G(\bC)$ in the literature.\par
We will make use of the following  result due to Bellamy-Schedler \cite[Theorem 1.18]{BS}):
\begin{thrm}\label{thrmSymp}{\rm {(Bellamy-Schedler)}}
 $M(n,G)_\bC$ is a  complex variety  with algebraic symplectic singularities if $G=\mathrm{SL}_d$ and $n\ge 1$. 
\end{thrm}
By definition, a normal complex variety has (complex) algebraic symplectic singularities if there exists an algebraic symplectic 2-form on the smooth locus extending to an algebraic 2-form on any resolution of singularities. Recall that, by Beauville \cite{B}, holomorphic, and hence also complex algebraic, symplectic singularities are Gorenstein rational singularities. For our purposes, working over $\bC$ is an underkill, so next we argue that by looking at the coefficients of the polynomials defining our scheme and the symplectic form, we can restrict to $\bQ$:

\begin{lemma}\label{lemQ} If $G=\mathrm{SL}_d$ and $n\ge 1$, then $M(n,G)$ is a variety over $\bQ$  with Gorenstein rational singularities over $\bQ$.
\end{lemma}
\begin{proof}
By Theorem \ref{thrmSymp}, $M(n,G)_\bC$ is a normal variety over $\bC$. This implies that $M:=M(n,G)$ is a normal variety over $\bQ$, see \cite[\href{http://stacks.math.columbia.edu/tag/038L}{Tag 038L}]{St}. In particular, it is a variety over $\bQ$. 

Let $h:\tilde{M}\ra M$ be a strong resolution of singularities over $\bQ$ of $M$, that is, a proper birational morphism which is an isomorphism above the regular locus of $M$. 
To show that $M$ has rational singularities over $\bQ$, we will show that $R^ih_*\cO_{\tilde{M}}=0$ for $i>0$. Let $h_\bC:\tilde{M}_\bC\ra M_\bC$ be the base change of $h$ to $\bC$. Since regularity is preserved after base change to $\bC$,  $h_\bC$ is a strong resolution of singularities of $M_\bC$ over $\bC$.  Since $M_\bC\ra M$ is a base change of the flat morphism $\spec(\bC)\ra\spec(\bQ)$, it is also a flat morphism. By flat base change \cite[\href{http://stacks.math.columbia.edu/tag/02KE}{Tag 02KE}]{St}, $R^i(h_\bC)_*\cO_{\tilde{M}_\bC}=(R^ih_*\cO_{\tilde{M}})\otimes_\bQ\bC$ for all $i\ge 0$. By Theorem \ref{thrmSymp}, $R^i(h_\bC)_*\cO_{\tilde{M}_\bC}=0$ for $i>0$ since $M_\bC$ has rational singularities over $\bC$. It follows that $R^ih_*\cO_{\tilde{M}}=0$ for $i>0$, and hence $M$ has rational singularities over $\bQ$.

Finally, $M$ is Gorenstein since $M_\bC$ is, by \cite[\href{http://stacks.math.columbia.edu/tag/0C02}{Tag 0C02}]{St}.
\end{proof}


We recall now the essential features of the volume form $v_\Gamma$, following \cite[4.3]{AA}.  Firstly, there is an open subscheme $$R(n,G)^{o}\subset R(n,G)$$ defined over $\bZ$, invariant under the action of $G$ by conjugation, such that $R(n,G)^{o}_\bQ$ is a smooth variety over $\bQ$, and such that for any algebraically closed field $k$ of characteristic zero, $R(n,G)^{o}(k)$ consists of representations $\pi_1(\Sigma_n)\ra G(k)$ with image generating a Zariski dense subgroup. Moreover, the geometric quotient 
$$M(n,G)^{o}=R(n,G)^{o}_\bQ/ G_\bQ$$
exists and is a smooth $\bQ$-subvariety of $M(n,G)$. On $M(n,G)^{o}$ one has an algebraic 2-form. Over $\bC$, this is the classical 2-form constructed by Atiyah, Bott, and Goldman, which  extends to the whole smooth locus of $M(n,G)_\bC$ to give the symplectic structure from Theorem \ref{thrmSymp}. We do not repeat here the definition, but only point out that from the proof in \cite{BS} of Theorem \ref{thrmSymp}, one sees easily that this algebraic symplectic 2-form on the smooth locus of $M(n,G)_\bC$ is  defined over $\bQ$, hence it gives a 2-form
$$\eta_G\in \bigwedge^2\Omega^1_{M(n,G)^{sm}/\bQ},$$
on the smooth locus of $M(n,G)$.

We will now use:

\begin{prop}{{\rm}(}\cite[Corollary 3.10]{AA}{{\rm )}} Let $X$ be a variety with rational singularities over a non-archimedean local field of characteristic zero $K$, and let $\omega$ be a top differential form on $X^{sm}$. Then for any $A\subset X(K)$, the integral
$$
m(A)=\int_{A\cap X^{sm}(K)}|\omega|
$$
defines a Radon measure, i.e. finite on compact subsets.
\end{prop}

Here, if $\omega=gdx_1\wedge\ldots dx_n$  on a compact open subset $U$ of $X^{sm}(K)$ analytically diffeomorphic to an open subset of $K^n$, then $|\omega |(U)$ is by definition $\int_U|g|d\lambda$, where $|g|$ is the normalized absolute value on $K$ and $\lambda$ is the standard additive Haar measure on $K^n$. This definition extends uniquely to a non-negative Borel measure $|\omega|$ on $X^{sm}(K)$, see \cite[3.1]{AA}.

To use this proposition, consider  the base change of $M(n,G)$ to $K$.  Then   Lemma \ref{lemQ} holds  with $\bQ$ replaced by $K$. This  together with the previous proposition implies:

\begin{cor} 
\label{cor:Radon}
Let $n\ge 2$ and $G=\mathrm{SL}_d$. Let $$v_G=(\eta_G)^{\wedge \dim M(n,G) /2 }$$ be the top form on $M(n,G)^{sm}$ generated by $\eta_G$. Then for any $A\subset M(n,G)(K)$, the integral
$$m(A)=\int_{A \cap M(n,G)^{sm}(K)}	|v_{G}|
$$
defines a Radon measure.
\end{cor}


\begin{proof}[Proof of Theorem \ref{mainTec}]
Let $p$ be the residue field characteristic of $K$ and let $\mathcal{O}_K$ be its ring of integers. 
The group $G(K)$ has a natural analytic structure, which, by \cite[16.6.(i)]{dixsauman1999analytic}, implies 
that it is  an $\mathcal{O}_K$-analytic group in the sense of \cite[Definition~13.8]{dixsauman1999analytic}. 
The topological group $\Gamma$, therefore, can be given the structure of a $p$-adic analytic group by 
\cite[Theorem~13.23]{dixsauman1999analytic}; and hence, by \cite[Corollary~8.34]{dixsauman1999analytic}, 
it contains a finite index open {uniform pro-$p$ subgroup}, say $\Gamma_1$. 
See \cite[Definition~4.1]{dixsauman1999analytic} for the definition of uniform pro-$p$ group. Since the abscissas 
$\alpha_\Gamma$ and $\alpha_{\Gamma_1}$ are equal by \cite[Lemma~2.2]{LM}, the  representation zeta functions 
of $\Gamma$ and $\Gamma_1$ have the same domain of convergence. We may therefore assume that $\Gamma$  
is a uniform pro-$p$ group.

Consider now the natural map of $K$-analytic manifolds
$$
q : R(n,\Gamma)^o/\Gamma \ra M(n,G)^{sm}(K).
$$
The map $q$ is \'{e}tale onto the image, with finite fibres (cf.\ Lemma~\ref{lem:finite_fibres}) and by definition the volume form $v_\Gamma$ on the source of this map is the pull-back of $v_G$. We will show in Lemma \ref{lemBDD} below that the number of points in the fibres of $q$ is bounded. Hence by Corollary \ref{cor:Radon}, to prove Theorem \ref{mainTec} it is enough to find a compact subset  in $M(n,G)(K)$ containing the image under $q$ of $R(n,\Gamma)^o/\Gamma$.  

For this purpose, denote by 
$$
Q: R(n,G)(K)\ra M(n,G)(K)
$$
the map defined by the categorical quotient, and by 
$$
\Phi: G(K)^{\times 2n}\ra G(K)
$$ 
the map
$$(g_1,h_1,\ldots, g_n,h_n)\mapsto [g_1,h_1]\cdot\ldots\cdot[g_n,h_n].$$
Note that 
$$
R(n,G)(K)=\Phi^{-1}(1).
$$
We will show that the image
$$
Q(\Phi^{-1}(1)\cap \Gamma^{\times 2n})
$$
is compact and contains the image $q(R(n,\Gamma)^o/\Gamma)$. This will finish the proof.

Since the image of a compact set under a continuous map is compact, to show the first claim it suffices to show that $\Phi^{-1}(1)\cap \Gamma^{\times 2n}$ is compact. Indeed, by hypothesis, $\Gamma\subset G(K)$ is compact, and so  $\Gamma^{\times 2n}\subset G(K)^{\times 2n}$ is also compact. Since $\Phi^{-1}(1)$ is closed in $G(K)^{\times 2n}$, $\Phi^{-1}(1)\cap \Gamma^{\times 2n}$ is closed in $\Gamma^{\times 2n}$  and hence compact.

For the second claim, note that
\begin{equation}\label{eqRo}
R(n,\Gamma)^o=[R(n,G)^o(K)]\cap [\Phi^{-1}(1)\cap \Gamma^{\times 2n}].
\end{equation}
Hence
$$
q(R(n,\Gamma)^o/\Gamma)=Q(R(n,\Gamma)^o) 
$$ 
and the latter is contained in $Q(\Phi^{-1}(1)\cap \Gamma^{\times 2n})$.
\end{proof}

\begin{rmk}
As explained in the introduction, Corollary~\ref{corRat} now follows directly from Theorem~\ref{mainTec} and \cite{AA}. 
A similar argument also shows Proposition~\ref{prop:semisimple}. Indeed let $G$ be a connected, simply connected, 
semi-simple algebraic group defined over $\bQ$, and let $n\in\bZ_{>0}$ be such that the character variety 
$M(n,G)$ has rational singularities. Corollary~\ref{cor:Radon}, holds also in this more general setting and the proof of 
Theorem~\ref{mainTec} also carries over: the argument reducing the problem to the case in which $\Gamma$ is a 
uniform pro-$p$ group goes through unchanged, and the map $q$ has fibres of bounded cardinality too, as we shall 
see in Section~\ref{sec:bounded_fibres}. It follows that for all non-archimedean local fields $K$ of characteristic $0$ 
and all compact open $\Gamma\leq G(K)$, the abscissa $\alpha_\Gamma < 2n -2$. So, by \cite[Theorem~IV]{AA}, 
$\homo(\pi_1(\Sigma_n),G)$ has rational singularities.
\end{rmk}


\section{Boundedness of $\Gamma$-orbits in a $G(K)$-orbit}
\label{sec:bounded_fibres}

We address here the last technical issue left open in the proof of Theorem \ref{mainTec}. Let $G$ be a semi-simple 
algebraic group defined over $\bQ$ and let $K$ be a non-archimedean local field of characteristic zero.

\begin{lemma}\label{lemBDD}
Let $\Gamma$ be an open uniform pro-$p$ subgroup of $G(K)$. Then there exists $N_\Gamma>0$ such that the number of points in every fibre of the \'{e}tale map 
$$
q : R(n,\Gamma)^o/\Gamma \ra M(n,G)^{sm}(K).
$$  is smaller than $N_\Gamma$.
\end{lemma}
The rest of this section is dedicated to the proof of this lemma.
	\newcommand{\orbit}[2]{\mathcal{O}_{#1}{(#2)}}
	\newcommand{\localfield}{K}
	\newcommand{\con}[2]{#1.#2}
	\newcommand{\Z}{\mathbf{Z}}
				
Recall that $q$ is an \'etale map onto the image
between two analytic manifolds, which means that it is a local analytic diffeomorphism. We start by proving that 
the fibres of $q$ are indeed finite.
\begin{lm}
\label{lem:finite_fibres} 
The fibres of $q$ have finite cardinality.
\begin{proof}
Let $x \in R(n,G)^{o}(\localfield)$ and let $X$ be its $G(\localfield)$-orbit. The fibre of $q$ above $X$ is 
$(X \cap \Gamma^{\times 2n})/\Gamma$, where $\Gamma$ acts as a subgroup of $G(\localfield)$. By 
\cite[Lemma~4.8~(3)]{AA} $X$ is Zariski closed in $G(\localfield)$ and so $X \cap \Gamma^{\times 2n}$ is compact, 
which implies that the orbit space $(X \cap \Gamma^{\times 2n})/\Gamma$ is compact as well. This proves 
that the fibre in hand is finite because  $q$ is a local diffeomorphism, hence its fibres are discrete.
\end{proof}
\end{lm}	
Let $g,h \in G(\localfield)$, we write $\con{g}{h} = g h g^{-1}$. By extension, if 
$x = (x_1,y_1\dots,x_n,y_n) \in G(\localfield)^{\times 2n}$, we 
write $\con{g}{x} = (\con{g}{x_1},\con{g}{y_1}\dots,\con{g}{x_n},\con{g}{y_n})$.

Fix throughout $x = (x_1,y_1\dots,x_n,y_n)\in R(n,\Gamma)^{o}$ and let $X$ be its $G(\localfield)$-orbit 
as in the proof of the previous lemma. Let also $H_x = \overline {\langle x_1,y_1,\dots,x_n,y_n\rangle}$ be 
the subgroup of $\Gamma$ topologically generated by the entries of $x$, where the overline will denote closure from now on. 
Recall that $H_x$ is open in $\Gamma$ because $x\in R(n,\Gamma)^{o}$, and so it has finite index in $\Gamma$.

In order to show that $q$ has fibres of bounded cardinality, we shall show that $X \cap \Gamma^{\times 2n}$ splits in 
a fixed number of $\Gamma$-orbits. We start by finding a better description for this set. Indeed, we shall show 
that, for $g\in G(\localfield)$, the $2n$-tuple $\con{g}{x}$ lies in $\Gamma^{\times 2n}$ if and only if $g$ normalizes $\Gamma$. 
For this purpose, notice that $$\con{g}{x}\in\Gamma^{\times 2n}\iff \con{g}{H_x} = H_{\con{g}{x}}\leq \Gamma.$$ In other 
words it suffices to show 
that $g$ normalizes $\Gamma$ whenever we find an open subgroup $H\leq\Gamma$ such that $\con{g}{H}\leq \Gamma$. 
Before doing so, we recall some properties of uniform pro-$p$ groups that we shall need for this proof.

The group $\Gamma$ is uniform and therefore finitely generated by definition. Let 
$\lbrace \gamma_1, \dots ,\gamma_r\rbrace$ 
be a minimal topological generating set for $\Gamma$. By definition, a uniform pro-$p$ group 
 is powerful. So, by \cite[Proposition~3.7]{dixsauman1999analytic},
\[
\Gamma = \overline{ \langle\gamma_1\rangle}\cdots\overline{ \langle\gamma_r\rangle}.
\]
Hence, for $\gamma\in\Gamma$ there are $\lambda_1,\dots,\lambda_r\in\Z_p$ such that 
$\gamma = \gamma_1^{\lambda_1}\cdots\gamma_r^{\lambda_r}$. Here, for $m\in\Z_p$, 
we write $\gamma^{m}$ to mean the element $(\gamma^{m_1},\gamma^{m_2},\dots)$ 
in the pro-finite completion of $\langle \gamma\rangle$, where 
$m = (m_1, m_2,\dots)\in\Z_p = \varprojlim_{j\in\Z_{> 0}} \Z/p^j\Z$. By 
\cite[Theorem~4.9]{dixsauman1999analytic} the function $\Gamma \rightarrow \Z_p^r$ sending
$\gamma$ to  $(\lambda_1,\dots,\lambda_r)$ is a homeomorphism and it becomes an 
isomorphism of $\Z_p$-modules if $\Gamma$ is endowed with the $\Z_p$-module structure 
described in \cite[Section~4.3]{dixsauman1999analytic} where left multiplication is defined by 
$m \gamma = \gamma^m$.

The next lemma introduces a left- and right-translation invariant metric on the locally compact group $G(\localfield)$. 
Let 
\[
	V_i = \overline{\langle \gamma_1^{p^{i}},\dots,\gamma_r^{p^{i}}\rangle}\,\,\,(i\in\Z_{\geq 0}).
\] 
Then $\lbrace V_i\rbrace_{i\in\Z_{\geq 0}}$ is a  decreasing countable compact (open) base at the identity. 
Moreover, if $\mu$ is the Haar measure on $G(\localfield)$ such that $\mu(\Gamma) = 1$, 
$\mu(V_i) = p^{-ir}$ for all $i\in\Z_{\geq 0}$. In what follows we denote by $\lvert m\rvert_p$ the $p$-adic 
absolute value of $m\in\Z_p$ and with 
$\|(\lambda_1,\dots,\lambda_r)\|_p = \max_{i = 1,\dots, r} \lvert \lambda_i\rvert_p$ the $p$-adic norm of 
$(\lambda_1,\dots,\lambda_r)\in\Z_p^r$. In addition, if $A$ and $B$ are two sets, we denote by $A \triangle B 
= (A \cup B)\smallsetminus (A\cap B)$ the symmetric difference of $A$ and $B$.
\begin{lm}
\label{lem:metric}
The function
	\[
	\rho(a,b) = \sup_{i\in\Z_{\geq 0}}\mu(aV_i \triangle bV_i) 
	\]
defines a left- and right-translation invariant metric on $G(\localfield)$. Moreover, if 
$\gamma = \prod_{j = 1}^{r} \gamma_j^{\lambda_j}$  for $\lambda_1,\dots,\lambda_r\in\Z_p$, then 
$\rho(1, \gamma) =2 (p^{-1} \| (\lambda_1,\dots,\lambda_r)\|_p)^r$.
\begin{proof}
With the exception of the right-invariance, the first part is just \cite[Lemma~1]{str1974metrics}. To 
show right-invariance and the second part, we use the properties of the 
filtration $\lbrace V_i\rbrace_{i\in \Z_{\geq 0}}$. First of all, let $a,b\in G(\localfield)$. For all $i\in \Z_{\geq 0}$, 
$V_i$ is a subgroup of $G(\localfield)$, so
	\[
		aV_i\triangle bV_i = \begin{cases}
						aV_i \cup bV_i 	& \text{ if } ab^{-1}\notin V_i \\
						\emptyset		& \text{ otherwise}.
					\end{cases}
	\]
Thus $\mu(aV_i\triangle bV_i) = 2 \mu(V_i) = 2 p^{-r i}$ if $ ab^{-1}\notin V_i$ and $0$ otherwise. 
This implies that, for $g\in G(\localfield)$, $\mu(agV_i\triangle bgV_i) =
 \mu(aV_i\triangle bV_i)$, because $ab^{-1} = agg^{-1}b^{-1}$. Thus $\rho$ is 
 right-translation invariant too. The last observation also shows the last part of the statement, as
 we notice that $\gamma\in V_i$ if 
 $p^{-i}  \geq \| (\lambda_1,\dots,\lambda_r)\|_p$ and $\gamma\notin V_i$ otherwise.
\end{proof}
\end{lm}
Let $N$  be the normalizer of $\Gamma$ in $G(\localfield)$.
\begin{lm}
\label{lem:normalizer}
If $H$ is an open subgroup of $\Gamma$ and $g\in G(\localfield)$ is such that 
$\con{g}{H}\leq \Gamma$, then $g\in N$.
\begin{proof}
Since $\Gamma$ is uniform, it is homeomorphic to $\Z_p^r$ and hence we may assume that
$H$ is topologically generated by $\gamma_1^{\eta},\dots, \gamma_r^{\eta}$ for some $\eta\in\Z$. 
Moreover, by hypothesis, $\con{g}{H}\leq \Gamma$; so for each $i \in\lbrace 1,\dots, r\rbrace$
\[
(\con{g}{\gamma_i})^{\eta} = \gamma_1^{\lambda_{i1}}\cdots \gamma_r^{\lambda_{ir}}, 
\] 
for some $\lambda_{i1},\dots,\lambda_{ir}\in\Z_p$.
The invariance under left- and right-translations of the metric $\rho$ implies that 
$\rho(1, \con{g}{(\gamma_i)^{\eta}} ) = \rho(1,\gamma_i^{\eta})$ so, by Lemma~\ref{lem:metric}, 
\[
	\vert \eta\rvert_p \allowbreak = \| (\lambda_{i1},\dots,\lambda_{ir})\|_p.
\] 
Hence 
$\eta^{-1}(\lambda_{i1},\dots,\lambda_{ir})\in \Z_p^{r}$ and therefore $\con{g}{\gamma_i} \in \Gamma$ because
the function associating an element of $\Gamma$ with its exponents for the generating set 
$\lbrace\gamma_1,\dots,\gamma_r\rbrace$ is a $\Z_p$-module isomorphism between the $\Z_p$-module $\Gamma$ 
and $\Z_p^r$.
\end{proof} 
\end{lm}
Lemma~\ref{lem:normalizer} shows that $g\in N$ whenever the conjugation by $g\in G(\localfield)$ sends  $H_x$ 
to a subgroup of $\Gamma$. This means exactly that for each $y \in X\cap \Gamma^{\times 2n}$, there is $h\in N$ such that 
$y = \con{h}{x}$, and hence that $X\cap \Gamma^{\times 2n} = \con{N}{x}$. The next lemma finishes the proof of 
Lemma~\ref{lemBDD}. Let 
$Z = Z(G(\localfield))$. Then $\Gamma Z\leq N$ because $\Gamma$ and $Z$ are both subgroups of $N$ and commute 
with each other. 
\begin{lm}
The fibres of $q$ have cardinality $[N:\Gamma Z]$.
\begin{proof}
We show that the set of $\Gamma$-orbits $(X \cap \Gamma^{\times 2n})/\Gamma$ is 
in bijection with the set of left $\Gamma Z$-cosets $N/\Gamma Z$. For $y\in X\cap \Gamma$, let $\orbit{\Gamma}{y}$ 
denote its $\Gamma$-orbit. We shall show that 
the following rule defines a bijective map:
	\[
		\varphi:\xymatrix@R=3pt{N/\Gamma Z	\ar[r] 			&(X \cap \Gamma^{\times 2n})/\Gamma\\
							h\Gamma Z	\ar@{|->}[r]	&\orbit{\Gamma}{h.x}.}
\]
We start by proving that $\varphi$ is well-defined. Let $h \in N$ and $\gamma z \in\Gamma Z$. Since $N$ normalizes
$\Gamma$, there exists $\gamma'\in \Gamma$ such that $h\gamma = \gamma' h$. Thus $\orbit{\Gamma}{\con{h}{x}} = 
\orbit{\Gamma}{\con{\gamma'h}{x}} = \orbit{\Gamma}{\con{h\gamma z}{x}}$ because $Z$ acts trivially on $\Gamma^{\times 2n}$.\par
The map $\varphi$ is surjective because $X\cap \Gamma = \con{N}{x}$ as Lemma~\ref{lem:normalizer} and
subsequent discussion show. We now show that $\varphi$ is injective. Take $h,h'\in N$ such that 
$\orbit{\Gamma}{\con{h}{x}}  = \orbit{\Gamma}{\con{h'}{x}}$. There is
$\gamma\in \Gamma$ such that 
	\[
		\gamma \con{h'}{x} = \con{h}{x}.
	\] 
Hence $h^{-1}\gamma h = z\in Z$ because by the proof of \cite[Lemma~4.8~(3)]{AA} the stabilizer of $x$ for the 
$G(\localfield)$-action on $R(n,G)^{o}(K)$ is $Z$. It follows that $\gamma h' = h z = z h$ and so $h^{-1}h' = \gamma^{-1} z\in \Gamma Z$, 
showing that $\varphi$ is injective.
\end{proof}
\end{lm}

\newpage

\begin{center}
{\bf Erratum: ``On representation zeta functions of special linear groups"}\\
Nero Budur and Michele Zordan
\end{center}

\bigskip

The proof of Theorem 2.2 of the article [BZ] contains a gap, as it was brought to our attention by a referee. More precisely Lemma 4.1 in the $p$-adic group theoretic argument is faulty. The fault was generated in line 4 of the proof of Lemma 4.3, and it propagated to the subsequent lemmas. Unfortunately this invalidates the proof of all the results claimed in the Introduction. We apologize for this mistake and thank the referee for pointing it out to us.

In the new preprint [B], an entirely geometric proof of Corollary 1.3 of [BZ] is given. This implies that all the claims in the Introduction of [BZ] are indeed still true, except Proposition 1.5 which remains open.

\section*{References}

\noindent [B] \;\,N. Budur, Rational singularities, quiver moment maps, and representations of surface groups. arXiv:1809.05180.

\medskip
\noindent[BZ]\; N. Budur, M. Zordan, On representation zeta functions of special linear groups, Int. Math. Res. Notices (2018), doi:10.1093/imrn/rny057.

\end{document}